\documentclass[12pt]{amsart}

\def\NZQ{\mathbb}               
\def\NN{{\NZQ N}}

\def\ZZ{{\NZQ Z}}
\def\RR{{\NZQ R}}

\def\frk{\mathfrak}               

\def\Phi{{\frk N}}
\def\opn#1#2{\def#1{\operatorname{#2}}} 
\opn\chara{char} \opn\length{\ell} \opn\pd{pd} \opn\rk{rk}
\opn\projdim{proj\,dim} \opn\injdim{inj\,dim} \opn\rank{rank}
\opn\depth{depth} \opn\grade{grade} \opn\height{height}
\opn\embdim{emb\,dim} \opn\codim{codim}

\opn\Tr{Tr} \opn\bigrank{big\,rank}
\opn\superheight{superheight}\opn\lcm{lcm}
\opn\trdeg{tr\,deg}
\opn\reg{reg} \opn\lreg{lreg} \opn\ini{in} \opn\lpd{lpd}
\opn\size{size}\opn{\mult}{mult}
\opn\div{div} \opn\Div{Div} \opn\cl{cl} \opn\Cl{Cl}
\opn\Spec{Spec} \opn\Supp{Supp} \opn\supp{supp} \opn\Sing{Sing}
\opn\Ass{Ass} \opn\Min{Min}
\opn\Ann{Ann} \opn\Rad{Rad} \opn\Soc{Soc}
\opn\Syz{Syz} \opn\Im{Im} \opn\Ker{Ker} \opn\Coker{Coker}
\opn\Am{Am} \opn\Hom{Hom} \opn\Tor{Tor} \opn\Ext{Ext}
\opn\End{End} \opn\Aut{Aut} \opn\id{id}

\opn\nat{nat}
\opn\pff{pf}
\opn\Pf{Pf} \opn\GL{GL} \opn\SL{SL} \opn\mod{mod} \opn\ord{ord}
\opn\Gin{Gin}
\opn\Hilb{Hilb}\opn\adeg{adeg}\opn\std{std}\opn\ip{infpt}
\opn\Pol{Pol}
\opn\sat{sat}
\opn\Var{Var}
\opn\Gen{Gen}
\opn\vol{vol}

\opn\aff{aff} \opn\con{conv} \opn\relint{relint} \opn\st{st}
\opn\lk{lk} \opn\cn{cn} \opn\core{core} \opn\vol{vol}
\opn\link{link} \opn\star{star}
\opn\gr{gr}

\def\Pc{{\mathcal P}}
\def\Qc{{\mathcal Q}}
\def\pot#1#2{#1[\kern-0.28ex[#2]\kern-0.28ex]}

\opn\dirlim{\underrightarrow{\lim}}
\opn\inivlim{\underleftarrow{\lim}}
\def\Implies{\ifmmode\Longrightarrow \else
        \unskip${}\Longrightarrow{}$\ignorespaces\fi}
\def\implies{\ifmmode\Rightarrow \else
        \unskip${}\Rightarrow{}$\ignorespaces\fi}
\def\iff{\ifmmode\Longleftrightarrow \else
        \unskip${}\Longleftrightarrow{}$\ignorespaces\fi}

\let\:=\colon
\newtheorem{Theorem}{Theorem}[section]

\newtheorem{Corollary}[Theorem]{Corollary}
\newtheorem{Proposition}[Theorem]{Proposition}
\newtheorem{Remark}[Theorem]{Remark}

\newtheorem{Example}[Theorem]{Example}
\newtheorem{Examples}[Theorem]{Examples}

\newtheorem{Problem}[Theorem]{Problem}

\let\epsilon\varepsilon
\let\phi=\varphi
\let\kappa=\varkappa
\def\qed{\ifhmode\textqed\fi
      \ifmmode\ifinner\quad\qedsymbol\else\dispqed\fi\fi}
\def\textqed{\unskip\nobreak\penalty50
       \hskip2em\hbox{}\nobreak\hfil\qedsymbol
       \parfillskip=0pt \finalhyphendemerits=0}
\def\dispqed{\rlap{\qquad\qedsymbol}}

\opn\dis{dis}
\def\pnt{{\raise0.5mm\hbox{\large\bf.}}}

\opn\Lex{Lex}

\begin{document}
\title{  Hermite normal forms and $\delta$-vector}
\author{Takayuki Hibi, Akihiro Higashitani and Nan Li}
\address{Takayuki Hibi,
Department of Pure and Applied Mathematics,
Graduate School of Information Science and Technology,pe
Osaka University,
Toyonaka, Osaka 560-0043, Japan}
\email{hibi@math.sci.osaka-u.ac.jp}
\address{Akihiro Higashitani,
Department of Pure and Applied Mathematics,
Graduate School of Information Science and Technology,
Osaka University,
Toyonaka, Osaka 560-0043, Japan}
\email{sm5037ha@ecs.cmc.osaka-u.ac.jp}
\address{Nan Li,
Department of Mathematics,
Massachusetts Institute of Technology,
Cambridge, MA 02139, USA}
\email{nan@math.mit.edu}
\begin{abstract}
Let $\delta(\Pc) = (\delta_0, \delta_1, \ldots, \delta_d)$ be the
$\delta$-vector of an integral polytope $\Pc \subset \RR^N$ of
dimension $d$. Following the previous work of characterizing the
$\delta$-vectors with $\sum_{i=0}^d \delta_i \leq 3$, the possible
$\delta$-vectors with $\sum_{i=0}^d \delta_i = 4$ will be
classified. And each possible $\delta$-vectors can be obtained by
simplices. We get this result by studying the problem of classifying
the possible integral simplices with a given $\delta$-vector
$(\delta_0, \delta_1, \ldots, \delta_d)$, where $\sum_{i=0}^d
\delta_i \leq 4$, by means of Hermite normal forms of square
matrices.
\end{abstract}

\maketitle

\section{introduction}

\subsection{Background for $\delta$-vectors}
Let $\Pc \subset \RR^N$ be an integral polytope of dimension $d$ and
$\partial \Pc$ its boundary. Define the numerical functions
$i(\Pc,n)$ and $i^*(\Pc,n)$ by setting
\[
i(\Pc,n) = |n\Pc \cap \ZZ^N|, \, \, \, \, \, \, \, \, \, \,
i^*(\Pc,n) = |n(\Pc - \partial \Pc) \cap \ZZ^N|.
\]
Here $n\Pc = \{ n\alpha : \alpha \in \Pc \}$ and $|X|$ is the
cardinality of a finite set $X$.
The systematic study of $i(\Pc,n)$ and $i^*(\Pc,n)$ originated in
Ehrhart \cite{Ehrhart} around 1955, who established the following
fundamental properties:
\begin{enumerate}
\item[(0.1)]
$i(\Pc,n)$ is a polynomial in $n$ of degree $d$;
\item[(0.2)]
$i(\Pc,0) = 1$;
\item[(0.3)]
(reciprocity law) $i^*(\Pc,n) = ( - 1 )^d i(\Pc, - n)$ for every
integer $n
> 0$.
\end{enumerate}
We say that $i(\Pc,n)$ is the {\em Ehrhart polynomial} of $\Pc$. An
introduction to Ehrhart polynomials is discussed in \cite[pp.
235--241]{StanleyEC} and \cite[Part II]{HibiRedBook}.

We define the sequence $\delta_0, \delta_1, \delta_2, \ldots$ of
integers by the formula
\begin{align}\label{delta}
(1 - \lambda)^{d + 1} \bigg( 1 + \sum_{n=1}^{\infty}
i(\Pc,n) \lambda^n \bigg)= \sum_{i=0}^{\infty} \delta_i \lambda^i.
\end{align}
In particular, $\delta_0 = 1$ and $\delta_1 = |\Pc \cap \ZZ^N| - (d +
1)$. Thus, if $\delta_1=0$, then $\Pc$ is a simplex. The
above facts (0.1) and (0.2) together with a well-known result on
generating function (\cite[Corollary 4.3.1]{StanleyEC}) guarantee
that $\delta_i = 0$ for every $i > d$. We say that the sequence
\[
\delta(\Pc) = (\delta_0, \delta_1, \ldots, \delta_d)
\]
which appears in Eq. \eqref{delta} is the {\em $\delta$-vector} of
$\Pc$ and the polynomial
\[
\delta_{\Pc}(t)=\delta_0+\delta_1t+\cdots+\delta_dt^d
\]
which also appears in Eq. \eqref{delta} is the {\em
$\delta$-polynomial} of $\Pc$.

It follows from the reciprocity law (0.3) that
\begin{align*}
(1 - \lambda)^{d + 1} \bigg( \sum_{n=1}^{\infty}
i^*(\Pc,n) \lambda^n \bigg) = \sum_{i=0}^{d} \delta_{d-i}
\lambda^{i+1}.
\end{align*}
In particular, $\delta_d = |(\Pc - \partial \Pc) \cap \ZZ^N|$. Each
$\delta_i$ is nonnegative (\cite{StanleyDRCP}).
If $\delta_d \neq 0$, then $\delta_1 \leq \delta_i$
for every $1 \leq i < d$ (\cite{HibiLBT}).

Let $s = \max\{ \, i \, : \, \delta_i \neq 0 \, \}$. Stanley
\cite{StanleyJPAA} shows that
\begin{align}\label{Stanley}
\delta_0 + \delta_1 + \cdots + \delta_i \leq
\delta_s + \delta_{s-1} + \cdots + \delta_{s-i}, \, \, \, \, \, 0
\leq i \leq \lfloor s/2 \rfloor
\end{align}
by using Cohen--Macaulay rings. The inequalities
\begin{align}\label{Hibi}
\delta_{d-1} + \delta_{d-2} + \cdots + \delta_{d-i}
\leq \delta_2 + \delta_3 + \cdots + \delta_i + \delta_{i+1}, \, \,
\, \, \, 1 \leq i \leq \lfloor (d-1)/2 \rfloor
\end{align}
appear in \cite[Remark (1.4)]{HibiLBT}.

\subsection{Main result: characterization of $\delta$-vectors with $\sum_{i=0}^{d} \delta_i =4$}
One of the most fundamental problems of enumerative combinatorics is
to find a combinatorial characterization of all vectors that can be
realized as the $\delta$-vector of some integral polytope.
For example, restrictions like $\delta_0=1$, $\delta_i\ge 0$,
\eqref{Stanley} and \eqref{Hibi} are necessary conditions for a
vector to be a $\delta$-vector of some integral polytope.

On the one hand, the complete classification of the $\delta$-vectors for dimension 2
is given essentially by Scott \cite{Scott},
while the case where the dimension is greater than or equal to 3
is presumably unknown. In \cite{smallvolume}, on the other hand,
the possible $\delta$-vectors with $\sum_{i=0}^{d} \delta_i \leq 3$
are completely classified by the inequalities \eqref{Stanley} and \eqref{Hibi}.
\begin{Theorem}\label{HHN}\cite[Theorem 0.1]{smallvolume}
Let $d \geq 3$. Given a sequence $(\delta_0,\delta_1,\ldots,\delta_d)$
of nonnegative integers, where $\delta_0=1$ and $\delta_1 \geq \delta_d$,
which satisfies $\sum_{i=0}^{d} \delta_i \leq 3$,
there exists an integral polytope $\Pc \subset \RR^d$ of dimension $d$
whose $\delta$-vector coinsides with $(\delta_0,\delta_1,\ldots,\delta_d)$
if and only if $(\delta_0,\delta_1,\ldots,\delta_d)$ satisfies
all inequalities \eqref{Stanley} and \eqref{Hibi}.
\end{Theorem}

However, this is not true for $\sum_{i=0}^{d} \delta_i =4$
(\cite[Example 1.2]{smallvolume}). In this paper, we will give the
complete classification of the possible $\delta$-vectors with
$\sum_{i=0}^{d} \delta_i =4$ (see Theorem \ref{case4}). Moreover,
similarly to $\sum_{i=0}^{d} \delta_i \leq 3$, it turns out that all
the possible $\delta$-vectors with $\sum_{i=0}^{d} \delta_i =4$ can
be chosen to be integral simplices, although this does not hold when
$\sum_{i=0}^{d} \delta_i =5$ (Remark \ref{simplex}).

\subsection{Approach: a classification of integral simplices with a given $\delta$-vector}

Let $\ZZ^{d \times d}$ denote the set of $d \times d$ integral matrices.
Recall that a matrix $A \in \ZZ^{d \times d}$ is {\em unimodular} if $\det (A) = \pm 1$.
Given integral polytopes $\Pc$ and $\Qc$ in $\RR^d$ of dimension $d$,
we say that $\Pc$ and $\Qc$ are {\em unimodularly equivalent}
if there exists a unimodular matrix $U \in \ZZ^{d \times d}$
and an integral vector $w$, such that $\Qc=f_U(\Pc)+w$,
where $f_U$ is the linear transformation in $\RR^d$ defined by $U$,
i.e., $f_U({\bf v}) = {\bf v} U$ for all ${\bf v} \in \RR^d$.
Clearly, if $\Pc$ and $\Qc$ are unimodularly equivalent, then
$\delta(\Pc) = \delta(\Qc)$. Conversely, given a vector $v\in
\ZZ_{\ge 0}^{d+1}$, it is natural to ask what are all the integral
polytopes $\Pc$ under unimodular equivalence, such that
$\delta(\Pc)=v$.

In this paper, we will focus on this problem for simplices with one
vertex at the origin. In addition, we do not allow any shifts in the
equivalence, i.e., $d$-dimensional integral polytopes $\Pc$ and
$\Qc$ are equivalent if there exists a unimodular matrix $U$, such
that $\Qc=f_U(\Pc)$. By considering the $\delta$-vectors of all the
integral simplices up to this equivalence whose normalized volumes
are 4, we will get our main result Theorem \ref{case4}.

For discussing the representative under equivalence of
the integral simplices with one vertex at the origin,
we consider Hermite normal forms.

Let $\Pc$ be an integral simplex in $\RR^d$ of
dimension $d$ with the vertices ${\bf 0}, {\bf v}_1, \ldots, {\bf v}_d$.
Define $M(\Pc) \in \ZZ^{d \times d}$ to be the matrix with the row vectors
${\bf v}_1, \ldots, {\bf v}_d$. Then we have the following connection
between the matrix $M(\Pc)$ and the $\delta$-vector of $\Pc$:
$|\det(M(\Pc))|=\sum_{i\ge 0}\delta_i$.
In this setting, $\Pc$ and $\Pc'$ are equivalent if and only if
$M(\Pc)$ and $M(\Pc')$ have the same Hermite normal form,
where the \emph{Hermite normal form} of a nonsingular integral square matrix $B$ is
the unique nonnegative lower triangular matrix
$A = (a_{ij}) \in \ZZ_{\ge 0}^{d \times d}$ such that $A = BU$
for some unimodular matrix $U \in \ZZ^{d \times d}$ and
$0 \leq a_{ij} < a_{ii}$ for all $1 \leq j < i$, (see \cite[Chapter 4]{Sch}).
In other words, we can pick the Hermite normal form as the representative
in each equivalence class and study the following

\begin{Problem}\label{normalform}
{\em Given a vector $v\in \ZZ_{\ge 0}^{d+1}$, classify all possible
$d \times d$ matrices $A \in \ZZ^{d \times d}$ which are in Hermite
normal form with $\delta(\Pc) = (\delta_0, \delta_1, \ldots,
\delta_d)=v$, where $\Pc \subset \RR^d$ is the integral simplex
whose vertices are the row vectors of $A$ together with the origin
in $\RR^d$. }
\end{Problem}


%
%

\subsection{Structure of this paper}
In Section 2, the way we approach Problem \ref{normalform} will be described.
Concretely, we develop an algorithm for any Hermite normal form $A$
to compute its $\delta$-vector. (See Theorem \ref{algo}.)
This actually gives a new way to compute the $\delta$-vector
for any integral simplex via its Hermite normal form.
This algorithm can be very efficient for simplices
with small volumes and prime volumes.

Based on this algorithm, as a by-product, we can derive some
conditions for Hermite normal forms to have ``shifted symmetric''
$\delta$-vector, namely, $\delta_i=\delta_{d+1-i}$ for $1\le i\le d$.
We will discuss these conditions for two classes of Hermite normal forms in Section 3.

In Section 4, we apply Theorem \ref{algo} and obtain
a solution to Problem \ref{normalform} when $\sum_{i=0}^d \delta_i \leq 4$.
Section 4.1 is devoted to studying the case $\sum_{i=0}^d \delta_i = 2$,
Section 4.2 is $\sum_{i=0}^d \delta_i = 3$ and
Section 4.3 is $\sum_{i=0}^d \delta_i = 4$.

Finally, in section 5, as our main result,
we show that the inequalities \eqref{Stanley} and \eqref{Hibi}
with an additional condition will give all possible $\delta$-vectors
with $\sum_{i=0}^d \delta_i = 4$. And in this case,
all $\delta$-vectors can be obtained by simplices (Theorem \ref{case4}).

\subsection*{Acknowledgements.} We thank Richard Stanley for useful
discussions and Steven Sam for carefully reading the previous draft
of this article. We also thank the reviewers for their very helpful
suggestion.

\section{An algorithm for the computation of
the $\delta$-vector of a simplex}

In this section, we introduce an algorithm for calculating the $\delta$-vector
of integral simplices arising from Hermite normal forms.

Let $M \in \ZZ^{d \times d}$. We write $\Pc(M)$ for the integral
simplex whose vertices are the row vectors of $M$ together with the
origin in $\RR^d$. We will present an algorithm to compute the
$\delta$-vector of $\Pc(M)$. To make the notation clear, we assume
$d=3$. The general case is completely analogous. Let $A$ be the
Hermite normal form of $M$. We have that $\{\Pc(M)\cap \ZZ^d\}$ is
in bijection with $\{\Pc(A)\cap \ZZ^d\}$. By definition,
\begin{align*}
A=
\begin{pmatrix}
a_{11} &0      &0     \\
a_{21} &a_{22} &0     \\
a_{31} &a_{32} &a_{33}
\end{pmatrix},
\end{align*}
where each $a_{ij}$ is a nonnegative integer.

For a vector $\lambda=(\lambda_1,\lambda_2,\lambda_3)$, consider
$$b(\lambda):=(\lambda_1,\lambda_2,\lambda_3)A=
(a_{11}\lambda_1+a_{21}\lambda_2+a_{31}\lambda_3,
a_{22}\lambda_2+a_{32}\lambda_3,a_{33}\lambda_3).$$
Then it is clear that the set of interior points inside $\Pc(A)$ ($ (\Pc(A) - \partial \Pc(A))\cap \ZZ^3 $)
is in bijection with the set
$$\{(\lambda_1,\lambda_2,\lambda_3)\mid \lambda_i>0,\,\lambda_1+\lambda_2+\lambda_3<1 ,\,b(\lambda)\in \ZZ^3\}.$$
An observation is that $ n(\Pc(A) - \partial \Pc(A))\cap \ZZ^3 $, for any $n\in \NN$, is in bijection with
$$\{(\lambda_1,\lambda_2,\lambda_3)\mid \lambda_i>0,\,\lambda_1+\lambda_2+\lambda_3<n ,\,b(\lambda)\in \ZZ^3\}.$$

We first consider all positive vectors $\lambda$ satisfying
$b(\lambda)\in \ZZ^3$. By the lower triangularity of the Hermite
normal form, we can start from the last coefficient of $b(\lambda)$
and move forward. It is not hard to see that each vector $\lambda$
should have the following form: ($\{r\}$ is the fractional part of a
rational number $r$.)
$$
\lambda_3=\lambda_3^{k,k_3}:=\frac{k}{a_{33}}+k_3, $$
$$\lambda_2=\lambda_2^{jk,k_2}:=\frac{j-\{a_{32}\lambda_3^k\}}{a_{22}}+k_2, $$ and
$$\lambda_1=\lambda_1^{ijk,k_1}:=\frac{i-\{a_{21}\lambda_2^{jk}+a_{31}\lambda_3^k\}}{a_{11}}+k_1,$$
for some nonnegative integers $k_3,k_2,k_1$, where
$k\in\{1,2,\ldots,a_{33}\}$, $j\in\{1,2,\ldots,a_{22}\}$,
$i\in\{1,2,\ldots,a_{11}\}$ and
$\lambda_1^{ijk}=\lambda_1^{ijk,0},\lambda_2^{jk}=\lambda_2^{jk,0},\lambda_3^k=\lambda_3^{k,0}$.
We call all the vectors $\lambda$ with the same index $(i,j,k)$
the \textit{congruence class of $(i,j,k)$}.

Now we go to the condition $\lambda_1+\lambda_2+\lambda_3<n$ in the
above bijection. As $n$ increases, we ask when is the first time
that a congruence class $(i,j,k)$ starts to produce interior points
inside $n\Pc(A)$. In other words, fix $(i,j,k)$. We want the
smallest $n$ such that $\lambda_1+\lambda_2+\lambda_3<n$ with
$\lambda_1,\,\lambda_2,\,\lambda_3>0$. Then it is clear that this
happens when $k_1=k_2=k_3=0$  and
$$n=s_{ijk}:=\lfloor\lambda_1^{ijk}+\lambda_2^{jk}+\lambda_3^k\rfloor+1, $$
where $\lfloor r\rfloor $ for a rational number is the biggest integer not larger than $r$.

Finally, when $n$ grows larger than $s_{ijk}$, we want to consider
how many interior points this fixed congruence class produces. Let
$n=s_{ijk}+\ell$, so each interior point corresponds to a choice of
$k_1\ge 0,k_2\ge 0,k_3\ge 0$ in the formula of
$\lambda_1^{ijk,k_1}$, $\lambda_2^{ij,k_2}$ and $\lambda_3^{i,k_3}$
such that $k_1+k_2+k_3\le \ell$. There are
$\bigg(\binom{d+1}{\ell}\bigg)=\binom{d+\ell}{\ell}$ choices in
total.

To sum up, we have the following two observations for each
congruence class $(i,j,k)$, $k\in\{1,2,\ldots,a_{33}\}$,
$j\in\{1,2,\ldots,a_{22}\}$, $i\in\{1,2,\ldots,a_{11}\}$:
\begin{enumerate}
\item $s_{ijk}$ is the smallest $n$ such that this congruence  class contributes interior points in the $n$-th dilation of $\Pc(A)$;
\item In the $(s_{ijk}+\ell)$-th dilation of $\Pc(A)$, this congruence class contributes $\bigg(\binom{d+1}{\ell}\bigg)$ interior points.
\end{enumerate}

Therefore, the following Theorem holds. We state it for a general
dimension $d$, and the proof is analogous to the case $d=3$.
\begin{Theorem}\label{algo}Let $\Pc(A)$ be a $d$-dimensional simplex
corresponding to a $d\times d$ matrix $A=(a_{ij}) \in \ZZ^{d \times d}$.
Then the generating function for the interior points of $n\Pc(A)$,
$i^*(\Pc(A),n) = |n(\Pc(A) -
\partial \Pc(A)) \cap \ZZ^d|$ is
$$\sum_{n=1}^{\infty}i^*(\Pc(A),n)t^n=
(1-t)^{-(d+1)}\sum_{\substack{(i_1,\dots,i_d)\\1\le i_j \le a_{ij}}}t^{s_{i_1\cdots i_d}},$$
where $$s_{i_1\cdots
i_d}=\bigg\lfloor\sum_{k=1}^{d}\lambda_{k}^{i_k,i_{k+1},\dots
i_d}\bigg\rfloor+1,$$ with $$\lambda_d^{i_d}=\frac{i_d}{a_{dd}},$$
and for each $1\le k<d$, $$\lambda_{k}^{i_k,i_{k+1},\dots
i_d}=a_{kk}^{-1}\bigg(i_k-\bigg\{\sum_{h=k+1}^{d}a_{hk}\lambda_h^{i_hi_{h+1}\dots
i_d}\bigg\}\bigg).$$

By the reciprocity law (0.3), we have
$$
\delta_{\Pc(A)}(t)=\sum_{\substack{(i_1,\dots,i_d)\\1\le i_j \le
a_{ij}}}t^{d+1-s_{i_1\dots i_d}}.
$$
\end{Theorem}

\begin{Example}{\em
Let $A$ be the $4 \times 4$ matrix
\begin{align*}
\begin{pmatrix}
1 &0 &0 &0 \\
0 &1 &0 &0 \\
1 &1 &2 &0 \\
1 &0 &1 &3
\end{pmatrix}.
\end{align*}
Consider
\begin{align*}b(\lambda)=
(\lambda_1,\lambda_2,\lambda_3,\lambda_4)A
=(\lambda_1+\lambda_3+\lambda_4,\lambda_2+\lambda_3,2\lambda_3+\lambda_4,3\lambda_4).
\end{align*}
Denote
$$
\lambda_4^{j}=\frac{j}{3}, \text{ for }j=1,2,3,
\,\,\,\,\lambda_3^{ij}=\frac{i-\{\lambda_4^j\}}{2}, \text{ for
}i=1,2,$$
$$\lambda_2^{ij}=1-\{\lambda_3^{ij}\},\,\,\,\,\lambda_1^{ij}=1-\{\lambda_3^{ij}+\lambda_4^{j}\}$$
and
$$s_{ij}=1+\lfloor \lambda_1^{ij}+\lambda_2^{ij}+\lambda_3^{ij}+\lambda_4^{j}\rfloor.$$
Then we have
$$s_{11}=2,s_{21}=3,s_{12}=2,s_{22}=3,s_{13}=3,s_{23}=5,$$
$$
\delta_{\Pc(A)}(t)=
\sum_{i=1}^{3}\sum_{j=1}^{2}t^{d+1-s_{ij}}=1+3t^2+2t^3,
$$
and thus
$$
\delta(\Pc(A))=(1,0,3,2,0).
$$
}\end{Example}

\section{Shifted symmetric $\delta$-vectors}

In this section, we define shifted symmetric $\delta$-vectors and
study its conditions for some special Hermite normal forms. Results in this
section are direct applications of the algorithm developed in the
previous section (Theorem \ref{algo}). In \cite{shifted}, the second
author studied  shifted symmetric $\delta$-vectors without using the
algorithm.

We call a $\delta$-vector {\em shifted symmetric}, if
$\delta_i=\delta_{d+1-i}$, $1\le i\le d$. For example,
$(1,\,1,\,2,\,2,\,1,\,2,\,2,\,1)$ is shifted symmetric.

We want this definition because it simply arises from the algorithm
for the ``one row'' Hermite normal forms as discussed in the first
subsection. In the second subsection, we will consider a special
``one row'' Hermite normal form, which allows us to have better
results.

\subsection{``One row'' Hermite normal forms}

Consider all $d\times d$ matrices with determinant $D$
and the following Hermite normal forms for some $k\in\{1,2,\dots,d\}$.

\begin{align}\label{onegen}A_D=
\begin{pmatrix}
1 &       &  &  &  &       &  \\
  &\ddots &  &  &  &       &  \\
  &       &1 &  &  &       &  \\
a_1 &\cdots & a_{k-1} &D &  &       &  \\
  &       &  &  &1 &       &  \\
  &       &  &  &  &\ddots &  \\
  &       &  &  &  &       &1
\end{pmatrix},
\end{align}
where $a_1,\dots,a_{k-1}$ are nonnegative integers smaller than $D$ and all other terms are zero.
Let $d_j$  denote the number of $j$'s among these $a_{\ell}$'s, for $j=1,\ldots,D-1$.
Then we can simplify Theorem \ref{algo}
for these ``one row'' Hermite normal forms.

\begin{Corollary}\label{onegencor}
Let $M \in \ZZ^{d \times d}$ with $\det(M)=D$ and $\Pc(M)$ be the
corresponding integral simplex. If its Hermite normal form is of the
form as in \eqref{onegen}, then we have
$$\delta_{\Pc(M)}(t)=\sum_{i=1}^D t^{d+1-s_i},$$ where
\begin{align}\label{prime}
s_i=\bigg\lfloor\frac{i}{D}-\sum_{j=1}^{D-1}\bigg\{\frac{ij}{D}\bigg\}d_j\bigg\rfloor+d.
\end{align}
\end{Corollary}
\begin{proof}
Consider
$$b(\lambda)=(\lambda_1,\dots,\lambda_k,\dots,\lambda_d)A_D=(\lambda_1+a_1\lambda_k,\dots,\lambda_{k-1}+a_{k-1}\lambda_k,
D\lambda_k,\lambda_{k+1},\dots,\lambda_d).$$ Using notation from the
proof of Theorem \ref{algo}, we have, for $i=1,2,\dots,D$,
$$\lambda_k^i=\frac{i}{D},\,\,\lambda_{\ell}^i=1-\bigg\{a_{\ell}\frac{i}{D}\bigg\},\text{ for }\ell=1,\dots,k-1$$
and $$\lambda_{k+1}^i=\dots=\lambda_{d}^i=1.$$ Therefore,
$s_i=1+\lfloor\lambda_1^i+\dots+\lambda_d^i\rfloor=\bigg\lfloor\frac{i}{D}
-\sum_{j=1}^{D-1}\bigg\{\frac{ij}{D}\bigg\}d_j\bigg\rfloor+d.$
\end{proof}

Now we are going to deduce a symmetry property of the $\delta$-vectors
by using this Corollary.
\begin{Proposition}[Shifted symmetry for ``one row'']\label{shifted}
For a matrix $M \in \ZZ^{d \times d}$ with Hermite normal form \eqref{onegen},
we have $s_i+s_{D-i}=d+1$, for $i=1,\dots,D-1$,
which implies $\delta_i=\delta_{d+1-i}$  by reciprocity,
if and only if the following three conditions hold:
\begin{enumerate}
\item $\sum_{j=1}^{D-1}jd_j - 1$ is coprime with $D$;
\item $d_j=0$ for all $j$ which is not coprime with $D$;
\item $\sum_{j=1}^{D-1}d_j=d-1$.
\end{enumerate}
\end{Proposition}
\begin{proof}
Let us consider $s_i+s_{D-i}$.
For an integer $a$, let $\overline{a}$ denote the residue class in $\ZZ / D \ZZ$.
Then we have
\begin{align*}
s_i+s_{D-i}&=
\bigg\lfloor\frac{i}{D}-\sum_{j=1}^{D-1}\bigg\{\frac{ij}{D}\bigg\}d_j\bigg\rfloor +
\bigg\lfloor\frac{D-i}{D}-\sum_{j=1}^{D-1}\bigg\{\frac{(D-i)j}{D}\bigg\}d_j\bigg\rfloor + 2d \\
&=\bigg\lfloor \frac{i-\sum_{j=1}^{D-1} \overline{ij}d_j}{D} \bigg\rfloor +
\bigg\lfloor \frac{D-i-\sum_{j=1}^{D-1} \overline{(D-i)j}d_j}{D} \bigg\rfloor +2d.
\end{align*}
Since
\begin{align}\label{rational}\begin{cases}
i-\sum_{j=1}^{D-1} \overline{ij}d_j \equiv i \bigg(1 - \sum_{j=1}^{D-1}jd_j \bigg) \;\; (\mod D), \\
D-i-\sum_{j=1}^{D-1} \overline{(D-i)j}d_j \equiv (D-i) \bigg(1 - \sum_{j=1}^{D-1}jd_j \bigg) \;\; (\mod D),
\end{cases}\end{align}
if the condition (1) is not satisfied, then one has
\begin{align*}
s_i+s_{D-i} &= \frac{D-\sum_{j=1}^{D-1} \bigg(\overline{ij} + \overline{(D-i)j} \bigg) d_j}{D}+2d \\
&=2d+1-\sum_{j=1}^{D-1} \frac{\overline{ij} + \overline{(D-i)j}}{D}d_j \\
&\geq 2d+1 - \sum_{j=1}^{D-1}d_j \geq d+2 > d+1
\end{align*}
for some $i$ with $1 \leq i \leq D-1$.
Thus, the condition (1) is a necessary condition to be $s_i+s_{D-i}=d+1$ for all $i$.
On the contrary, when the condition (1) is satisfied, again from \eqref{rational}, we have
\begin{align*}
s_i+s_{D-i} &= \frac{D-\sum_{j=1}^{D-1} \bigg(\overline{ij} + \overline{(D-i)j} \bigg) d_j}{D}+2d-1 \\
&= 2d - \sum_{j=1}^{D-1} \frac{\overline{ij} + \overline{(D-i)j}}{D} d_j \\
&= 2d - \sum_{D \text{ does not divide } ij}d_j.
\end{align*}
If the condition (2) is not satisfied, then we have
$$s_i+s_{D-i}= 2d - \sum_{D \text{ does not divide } ij}d_j > d+1$$
for some $i$ with $1 \leq i \leq D-1$.
Hence, the condition (2) is also a necessary condition.
In addition, if the condition (3) is not satisfied, then we have $s_i+s_{D-i} > d+1$.
Thus, the condition (3) is also a necessary condition.
On the other hand, when the conditions (1), (2) and (3) are all satisfied,
we have $s_i+s_{D-i}=D+1$ for all $i$.

Therefore, we obtain a necessary and sufficient to be
$s_i+s_{D-i}=d+1$ for all $i$.
\quad\quad\quad\quad\quad\quad\quad\quad\quad
\quad\quad\quad\quad\quad\quad\quad\quad\quad
\quad\quad\quad\quad\quad\quad\quad\quad\quad
\end{proof}

The conditions of Proposition \ref{shifted} are not very easy to check,
so we consider a special case of Hermite normal forms \eqref{onegen}.

\subsection{``All $D-1$ one row''  Hermite normal forms}

Assume in addition $d_{D-1}=d-1$ in Corollary \ref{onegencor}, i.e.,
the Hermite normal form looks like
\begin{align}\label{all}
\begin{pmatrix}
1   &    &       &    &  \\
    &1   &       &    &  \\
    &    &\ddots &    &  \\
    &    &       &1   &  \\
D-1 &D-1 &\cdots &D-1 & D
\end{pmatrix}.
\end{align}
 Then we have
\begin{Corollary}[All $D-1$] For a matrix $M \in \ZZ^{d \times d}$ with Hermite normal
form \eqref{all}, we have
$$\delta_{\Pc(M)}(t)=\sum_{i=1}^{D} t^{d+1-s_i}, \text{ where } s_i=\bigg\lfloor\frac{id}{D}\bigg\rfloor+1.$$
\end{Corollary}

For the Hermite normal form \eqref{all},
the conditions for shifted symmetry in Proposition \ref{shifted} can be simplified.
\begin{Proposition}[Shifted symmetry for ``all $D-1$ one row'']
Let $M \in \ZZ^{d \times d}$ with Hermite normal form \eqref{all}. Then
\begin{enumerate}
\item $\delta_i=\delta_{d+1-i}$ if and only if $D$ and $d$ are coprime.
\item When $D=kd$, for $k\in \NN$ and $k\ge2$,
the $\delta$-vector is $$(1,\underbrace{k,\dots,k}_{d-1},k-1),$$ which is not shifted symmetric.
But for $k=2$, we have $\delta_k=\delta_{d-k}$ (Gorenstein).
\end{enumerate}
\end{Proposition}

\section{Classification of Hermite normal forms with a given $\delta$-vector}
In this section, we will give another application of the algorithm
Theorem \ref{algo}. Consider Problem \ref{normalform} first with the
assumption that matrix $A \in \ZZ^{d \times d}$ has prime
determinant, i.e., $A$ is of the form \eqref{onegen}, with only one
general row. By Corollary \ref{onegencor}, in order to classify all
possible  Hermite normal forms \eqref{onegen} with a given
$\delta$-vector $(\delta_0, \delta_1, \ldots, \delta_d)$, we need to
find all nonnegative integer solutions $(d_1,d_2,\dots,d_{D-1})$
with $d_1+d_2+\dots+d_{D-1}\le d-1$ such that
$$\#\{i:d+1-s_i=j, \text{ for }i=1,\dots, D\}=\delta_j, \text{ for }j=0,\dots,d.$$
By Corollary \ref{onegencor}, we can build equations with ``floor''
expressions for $(d_1,d_2,\dots,d_{D-1})$. Remove the ``floor''
expressions, we obtain $D$ linear equations of
$(d_1,d_2,\dots,d_{D-1})$ with different constant terms but the same
$D\times D$ coefficient matrix $M$ with $ij$ entry $\{(ij)\mod
\,D\}$, which is a number in $\{0,1,\dots,D-1\}$. Then we first get
all integer solutions $(d_1,d_2,\dots,d_{D-1})$, and then test every
candidates by the restrictions of nonnegativity and
$d_1+d_2+\dots+d_{D-1}\le d-1$.

For $D=2$ and $3$, the coefficient matrix $M$ is nonsingular, so we
can write down the complete solutions, as presented in the first two
subsections. For larger primes, the coefficient matrix becomes
singular, so there are free variables in the integer solutions
$(d_1,d_2,\dots,d_{D-1})$, which make it very hard to simplify the
final solutions after the test.

The idea is similar for Hermite normal forms with non prime
determinant. Instead of using Corollary \ref{onegencor}, we need to
use the formulas in Theorem \ref{algo}. In the third subsection, we
will present the complete solution for $D=4$.

\subsection{A solution of Problem \ref{normalform}
when $\sum_{i=0}^d\delta_i = 2$}

The goal of this subsection is to give a solution of Problem
\ref{normalform} when $\sum_{i=0}^d\delta_i = 2$, i.e., given a
$\delta$-vector $\delta(\Pc)$ with $\sum_{i=0}^d\delta_i = 2$, we
classify all the integral simplices with $\delta(\Pc)$ arising from
Hermite normal forms with determinant 2.

We consider all Hermite normal forms \eqref{onegen} with $D=2$, namely,
\begin{align}\label{2}
A_2=\begin{pmatrix}
1 &       &  &  &  &       &  \\
  &\ddots &  &  &  &       &  \\
  &       &1 &  &  &       &  \\
* &\cdots &* &2 &  &       &  \\
  &       &  &  &1 &       &  \\
  &       &  &  &  &\ddots &  \\
  &       &  &  &  &       &1
\end{pmatrix},
\end{align}
where there are $d_1$ 1's among the $*$'s. Notice that the position
of the row with a 2 does not affect the $\delta$-vector, so the only
variable is $d_1$. By Corollary \ref{onegencor}, we have a formula
for the $\delta$-vector of this integral simplex $\Pc(A_2)$. Denote
$$
k=1-\bigg\lfloor\frac{1-d_1}{2}\bigg\rfloor.
$$
Then one has $\delta_0=\delta_k=1$.

By this formula, we can characterize all Hermite normal forms with a
given $\delta$-vector. Let $\delta_0=\delta_i=1$. Then by solving
the equation $i=1-\left\lfloor\frac{1-d_1}{2}\right\rfloor$, we
obtain $d_1=2i-2$ and $d_1=2i-1$, both cases will give us the
desired $\delta$-vector.

Notice that there is a constraint on $d_1$ to be $0 \leq d_1 \leq
d-1$. Not all $\delta$-vectors are obtained by from simplices. But
we can easily get the restriction of $i$ and the corresponding $d_1$
as follows (by $d_1 \geq 0$, we have $i \geq 1$):

\begin{enumerate}
\item If $i \leq d/2$, $d_1=2i-2$ and $d_1=2i-1$ both work,
and these give all the matrices with this $\delta$-vector.
\item If $i=(d+1)/2$, only $d_1=2i-2=d-1$ works.
\item If $i > (d+1)/2$, there is no solution.
\end{enumerate}

Now, this result has been obtained essentially in \cite{smallvolume}.
In fact, the inequality $i \leq (d+1)/2$ means that
the $\delta$-vector satisfies \eqref{Hibi}.

\subsection{A solution of Problem \ref{normalform}
when $\sum_{i=0}^d\delta_i = 3$}

We consider all Hermite normal forms \eqref{onegen} with $D=3$, namely,
\begin{align}\label{3}
A_3=\begin{pmatrix}
1 &       &  &  &  &       &  \\
  &\ddots &  &  &  &       &  \\
  &       &1 &  &  &       &  \\
* &\cdots &* &3 &  &       &  \\
  &       &  &  &1 &       &  \\
  &       &  &  &  &\ddots &  \\
  &       &  &  &  &       &1
\end{pmatrix},
\end{align}
where there are $d_1$ 1's and $d_2$ 2's among the $*$'s. Since the
position of the row with a 3 does not affect the $\delta$-vector, so
the only variables are $d_1$ and $d_2$. Also, by Corollary
\ref{onegencor}, we have $\delta_{\Pc(A_3)}(t)=1+t^{k_1}+t^{k_2}$,
where
\begin{align*}
k_1=1-\bigg\lfloor\frac{1-d_1-2d_2}{3}\bigg\rfloor \;\text{and} \;\;
k_2=1-\bigg\lfloor\frac{2-2d_1-d_2}{3}\bigg\rfloor.
\end{align*}

Then by the formula, similar to the case of
$\sum_{i=0}^d\delta_i=2$, though a little more complicated, we can
characterize all Hermite normal forms with a given $\delta$-vector.
Let $\delta_{\Pc(A_3)}(t)=1+t^i+t^j$. Set
\begin{align*}
i=1-\left\lfloor\frac{1-d_1-2d_2}{3}\right\rfloor \;\text{and} \;\;
j=1-\left\lfloor\frac{2-2d_1-d_2}{3}\right\rfloor.
\end{align*}
(Later reverse the role of $i$ and $j$ if $i \not= j$, in both
equations and solutions.) After computations, the solutions for
$(d_1,d_2)$ are
\begin{align*}
d^{(1)}=
\begin{cases}
d_1=2j-i \\
d_2=2i-j-1,
\end{cases}
d^{(2)}=
\begin{cases}
d_1=2j-i-1 \\
d_2=2i-j-1
\end{cases}
\text{and} \;\;\; d^{(3)}=
\begin{cases}
d_1=2j-i \\
d_2=2i-j-2.
\end{cases}
\end{align*}
In addition, by the restriction on $(d_1,d_2)$ that
$d_1,d_2 \geq 0$ and $d_1+d_2 \leq d-1$,
we have the following characterizations:

\begin{table}[ht]\label{vol3}
\centering 
\caption{Characterizations for matrices of the form \eqref{3}}
\begin{tabular}{|c| c|c| c|} 
\hline 
$2j$ & $2i $ & $i+j $ & solutions \\
\hline \hline  
$ \ge i$& $ \ge j+1$ & $ \le d$ &$d^{(1)}$\\
 \hline
 $ \ge i+1$& $  \ge j+1$ & $\le d+1$ &$d^{(2)}$\\
 \hline
$ \ge i$& $ \ge j+2$ & $\le d+1$ &$d^{(3)}$\\
 \hline
\end{tabular}
\end{table}

\begin{enumerate}
\item If $2j \geq i, 2i \geq j+1$ and $i+j \leq d$,
then the solution $d^{(1)}$ will work and
this gives all the matrices with this $\delta$-vector.
\item If $2j \geq i+1, 2i \geq j+1$ and $i+j \leq d+1$,
then the solution $d^{(2)}$ will work and
this gives all the matrices with this $\delta$-vector.
\item If $2j \geq i, 2i \geq j+2$ and $i+j \leq d+1$,
then the solution $d^{(3)}$ will work and
this gives all the matrices with this $\delta$-vector.
\item If $\{i,j\}$ in the given vector does not satisfy any of the above cases,
there is no matrix with this vector as its $\delta$-vector.
\end{enumerate}

Again, this result has been obtained in \cite{smallvolume}.
In fact, for example, the inequality $2j \geq i$ means that \eqref{Stanley} holds
and the inequality $i+j \leq d+1$ means that \eqref{Hibi} holds.

Notice that only the solution
\begin{align*}
d^{(2)}=
\begin{cases}
d_1=d-1 \\
d_2=0
\end{cases}
\end{align*}
works when $i=(d+2)/3$ and $j=(2d+1)/3$.
This happens when $d \equiv 1 \; (\mod 3)$
and there is only one matrix with $d_1=d-1$ and $d_2=0$.
Similary, only the solution
\begin{align*}
d^{(3)}=
\begin{cases}
d_1=0 \\
d_2=d-1
\end{cases}
\end{align*}
works when $i=(2d+2)/3$ and $j=(d+1)/3$.
This happens when $d \equiv 2 \; (\mod 3)$
and again, there is only one matrix with $d_1=0$ and $d_2=d-1$.


\subsection{A solution of Problem \ref{normalform}\label{vol4}
when $\sum_{i=0}^d\delta_i = 4$} When the determinant is 4, there
are two cases of Hermite normal forms. One is the Hermite normal
forms \eqref{onegen} with $D=4$, namely,
\begin{align}\label{4}
A_4=\begin{pmatrix}
1 &       &  &  &  &       &  \\
  &\ddots &  &  &  &       &  \\
  &       &1 &  &  &       &  \\
* &\cdots &* &4 &  &       &  \\
  &       &  &  &1 &       &  \\
  &       &  &  &  &\ddots &  \\
  &       &  &  &  &       &1
\end{pmatrix},
\end{align}
where there are $d_1$ 1's, $d_2$ 2's and $d_3$ 3's among $*$'s. The
other one looks like
\begin{align}\label{22}
A_4'=\begin{pmatrix}
1 &       &  &      &  &       &  &  &  &        \\
  &\ddots &  &      &  &       &  &  &  &        \\
  &       &1 &      &  &       &  &  &  &        \\
* &\cdots &* &2     &  &       &  &  &  &        \\
  &       &  &      &1 &       &  &  &  &        \\
  &       &  &      &  &\ddots &  &  &  &        \\
  &       &  &      &  &       &1 &  &  &        \\
\dot{*}&\cdots &\dot{*}&\bar{*}   &\dot{*}&\cdots &\dot{*}&2 &  &        \\
  &       &  &      &  &       &  &  &1 &        \\
  &       &  &      &  &       &  &  &  &\ddots
\end{pmatrix},
\end{align}
where there are $d_1$ 1's (resp. $d_1'$ 1's) among $*$'s (resp.
$\dot{*}$'s), there are $e_1$ 1's (resp. $e_1'$ 1's) among $*$'s
(resp. $\dot{*}$'s) of which the entry of the row of $\dot{*}$ (resp. $*$)
in the same column is 0. Also, set $d_1''=e_1+e_1'$.
(For example, a $6 \times 6$ Hermite normal form
\begin{align*}
\begin{pmatrix}
1 &0 &0 &0 &0 &0 \\
0 &1 &0 &0 &0 &0 \\
0 &0 &1 &0 &0 &0 \\
0 &1 &1 &2 &0 &0 \\
0 &0 &0 &0 &1 &0 \\
1 &1 &0 &1 &1 &2 \\
\end{pmatrix}
\end{align*}
is a matrix \eqref{22} with $d_1=2,d_1'=3,e_1=1,e_1'=2,d_1''=3$ and
$\bar{*}=1$.)

First, we consider the Hermite normal forms \eqref{4}.
Then, by Corollary \ref{onegencor}, we have
$\delta_{\Pc(A_4)}(t)=1+t^{k_1}+t^{k_2}+t^{k_3}$, where
\begin{align*}
k_1=1-\bigg\lfloor\frac{1-d_1-2d_2-3d_3}{4}\bigg\rfloor,
k_2=1-\bigg\lfloor\frac{1-d_1-d_3}{2}\bigg\rfloor \; \text{and} \;
k_3=1-\bigg\lfloor\frac{3-3d_1-2d_2-d_3}{4}\bigg\rfloor.
\end{align*}

Let $\delta_{\Pc(A_4)}(t)=1+t^i+t^j+t^k$. We get three sets of
equations, according to the order of $k_1,k_2$ and $k_3$:
\begin{align*}
i=1-\left\lfloor\frac{1-d_1-2d_2-3d_3}{4}\right\rfloor,
j=1-\left\lfloor\frac{1-d_1-d_3}{2}\right\rfloor \;\text{and} \;
k=1-\left\lfloor\frac{3-3d_1-2d_2-d_3}{4}\right\rfloor.
\end{align*}
(Later replace the roles of $i,j$ and $k$ if any of the three are
distinct.) After computations, the solutions for $(d_1,d_2,d_3)$ are
\begin{align*}
&d^{(1)}=
\begin{cases}
d_1=-i+j+k-1 \\
d_2=i-2j+k \\
d_3=i+j-k-1,
\end{cases}
d^{(2)}=
\begin{cases}
d_1=-i+j+k \\
d_2=i-2j+k \\
d_3=i+j-k-2,
\end{cases}& \\
&d^{(3)}=
\begin{cases}
d_1=-i+j+k \\
d_2=i-2j+k \\
d_3=i+j-k-1
\end{cases}
\text{and} \;\;\; d^{(4)}=
\begin{cases}
d_1=-i+j+k \\
d_2=i-2j+k-1 \\
d_3=i+j-k-1.
\end{cases}&
\end{align*}
In addition, by the restriction on $(d_1,d_2,d_3)$ that
$d_1,d_2,d_3 \geq 0$ and $d_1+d_2+d_3 \leq d-1$,
we have the following characterizations:

\begin{table}[ht]\label{(9)}
\caption{Characterizations for matrices of the form \eqref{4}}
\centering 
\begin{tabular}{|c| c|c| c|} 
\hline 
$j+k$ & $2j $ & $i+j $ & solutions \\
\hline \hline  
$ \ge i+1$& $ \le i+k\le d+1$ & $ \ge k+1$ &$d^{(1)}$\\
 \hline
 $ \ge i$& $  \le i+k\le d+1$ & $\ge k+2$ &$d^{(2)}$\\
 \hline
$ \ge i$& $ \le i+k\le d$ & $\ge k+1$ &$d^{(3)}$\\
 \hline
$ \ge i$& $ \le i+k-1\le d$ & $\ge k+1$ &$d^{(4)}$\\
 \hline
\end{tabular}
\end{table}

\begin{enumerate}
\item If $j+k \geq i+1, 2j \leq i+k \leq d+1$ and $i+j \geq k+1$,
then the solution $d^{(1)}$ will work and
this gives all the matrices with this $\delta$-vector.
\item If $j+k \geq i, 2j \leq i+k \leq d+1$ and $i+j \geq k+2$,
then the solution $d^{(2)}$ will work and
this gives all the matrices with this $\delta$-vector.
\item If $j+k \geq i, 2j \leq i+k \leq d$ and $i+j \geq k+1$,
then the solution $d^{(3)}$ will work and
this gives all the matrices with this $\delta$-vector.
\item If $j+k \geq i, 2j+1 \leq i+k \leq d+1$ and $i+j \geq k+1$,
then the solution $d^{(4)}$ will work and
this gives all the matrices with this $\delta$-vector.
\item If $\{i,j,k\}$ in the given vector does not satisfy any of the above cases,
there is no matrix \eqref{4} with this vector as its $\delta$-vector.
\end{enumerate}

Notice that only the solution
\begin{align*}
d^{(2)}=
\begin{cases}
d_1=0 \\
d_2=0 \\
d_3=d-1
\end{cases}
\end{align*}
works when $i=(3d+3)/4,j=(d+1)/2$ and $k=(d+1)/4$. This happens when
$d \equiv 3 \; (\mod 4)$ and there is only one matrix with
$d_3=d-1$. Similarly, only the solution
\begin{align*}
d^{(1)}=
\begin{cases}
d_1=d-1 \\
d_2=0 \\
d_3=0
\end{cases}
\end{align*}
works when $i=(d+3)/4,j=(d+1)/2$ and $k=(3d+1)/4$.
This happens when $d \equiv 1 \; (\mod 4)$
and again, there is only one matrix with $d_1=d-1$.

Next, we consider the Hermite normal forms \eqref{22}. However, we
need to consider two cases, which are the cases where $\bar{*}=0$ and
$\bar{*}=1$.

First, we consider the case with $\bar{*}=0$. Notice that the
variables are $d_1,d_1'$ and $d_1''$. Obviously we cannot use
Corollary \ref{onegencor}, but we apply Theorem \ref{algo} directly.
Thus we have $\delta_{\Pc(A_4')}(t)=1+t^{k_1}+t^{k_2}+t^{k_3}$,
where
\begin{align*}
k_1=\bigg\lfloor\frac{d_1+2}{2}\bigg\rfloor, \;
k_2=\bigg\lfloor\frac{d_1'+2}{2}\bigg\rfloor \; \text{and} \;
k_3=\bigg\lfloor\frac{d_1''+3}{2}\bigg\rfloor.
\end{align*}

Let $\delta_{\Pc(A_4')}(t)=1+t^i+t^j+t^k$.
We get three sets of equations, according to the order of $k_1,k_2$ and $k_3$:
\begin{align*}
i=\bigg\lfloor\frac{d_1+2}{2}\bigg\rfloor, \;
j=\bigg\lfloor\frac{d_1'+2}{2}\bigg\rfloor \;\text{and}\;
k=\bigg\lfloor\frac{d_1''+3}{2}\bigg\rfloor.
\end{align*}
or replace the role of $i,j$ and $k$ if $i,j$ and $k$ are distinct,
in all equations and solutions. After computations, since
$d_1+d_1'+d_1''$ is even, the solutions of $(d_1,d_1',d_1'')$ are
\begin{align*}
&d^{(1)}=
\begin{cases}
d_1=2i-2 \\
d_1'=2j-1 \\
d_1''=2k-3,
\end{cases}
d^{(2)}=
\begin{cases}
d_1=2i-1 \\
d_1'=2j-2 \\
d_1''=2k-3,
\end{cases}& \\
&d^{(3)}=
\begin{cases}
d_1=2i-1 \\
d_1'=2j-1 \\
d_1''=2k-2
\end{cases}
\text{and} \;\;\; d^{(4)}=
\begin{cases}
d_1=2i-2 \\
d_1'=2j-2 \\
d_1''=2k-2.
\end{cases}&
\end{align*}

In addition, by the restriction on $(d_1,d_1',d_1'')$ that $0 \leq
d_1 \leq d-2$, $0 \leq d_1' \leq d-2$, $0 \leq d_1'' \leq d-2$,
$d_1+d_1'+d_1'' \leq 2(d-2)$, $d_1'' \leq d_1+d_1'$, $d_1' \leq
d_1+d_1''$ and $d_1 \leq d_1'+d_1''$, we have the following
characterizations:

\begin{table}[ht]\label{10-0}
\centering 
\caption{Characterizations for matrices of the form \eqref{22} with
$\bar{*}=0$}
\begin{tabular}{|c| c|c| c|c|c|c|c|} 
\hline 
$i$ & $j$ & $k$ & $i+j$& $i+k$ & $j+k$ & $i+j+k$ & solutions \\
\hline \hline  
$\le\bigg\lfloor\frac{d}{2}\bigg\rfloor$
&$\le\bigg\lfloor\frac{d-1}{2}\bigg\rfloor$
&$\geq 2$,
&$\ge k$ & $\ge j+2$ &$\ge i+1$ & $\le d+1$ & $d^{(1)}$ \\
 & &$\le\bigg\lfloor\frac{d+1}{2}\bigg\rfloor$ & & & & & \\
\hline $\le\bigg\lfloor\frac{d-1}{2}\bigg\rfloor$
&$\le\bigg\lfloor\frac{d}{2}\bigg\rfloor$
&$\geq 2$, &$\ge k$ & $\ge j+1$ &$\ge i+2$ & $\le d+1$ & $d^{(2)}$ \\
 & &$\le\bigg\lfloor\frac{d+1}{2}\bigg\rfloor$ & & & & & \\
\hline $\le\bigg\lfloor\frac{d-1}{2}\bigg\rfloor$
&$\le\bigg\lfloor\frac{d-1}{2}\bigg\rfloor$&
$\le\bigg\lfloor\frac{d}{2}\bigg\rfloor$ & $\ge k$ & $\ge j+1$
&$\ge i+1$ & $\le d$ & $d^{(3)}$\\
\hline $\le\bigg\lfloor\frac{d}{2}\bigg\rfloor$
&$\le\bigg\lfloor\frac{d}{2}\bigg\rfloor$&
$\le\bigg\lfloor\frac{d}{2}\bigg\rfloor$ & $\ge k+1$ & $\ge j+1$
&$\ge i+1$ & $\le d+1$ & $d^{(4)}$\\
\hline
\end{tabular}
\end{table}

\begin{enumerate}
\item If $i \leq \lfloor d/2 \rfloor,\;j \leq \lfloor (d-1)/2 \rfloor,\;
2 \leq k \leq \lfloor (d+1)/2 \rfloor,$
$i+j+k \leq d+1, k \leq i+j, j+2 \leq i+k$ and $i+1 \leq j+k$,
then the solution $d^{(1)}$ will work and
this gives all the matrices with this $\delta$-vector.
\item If $i \leq \lfloor (d-1)/2 \rfloor,\;j \leq \lfloor d/2 \rfloor,\;
2 \leq k \leq \lfloor (d+1)/2 \rfloor,$
$i+j+k \leq d+1, k \leq i+j, j+1 \leq i+k$ and $i+2 \leq j+k$,
then the solution $d^{(2)}$ will work and
this gives all the matrices with this $\delta$-vector.
\item If $k \leq \lfloor d/2 \rfloor,\;i,j \leq \lfloor (d-1)/2 \rfloor,$
$i+j+k \leq d, k \leq i+j, j+1 \leq i+k$ and $i+1 \leq j+k$,
then the solution $d^{(3)}$ will work and
this gives all the matrices with this $\delta$-vector.
\item If $i,j,k \leq \lfloor d/2 \rfloor,$
$i+j+k \leq d+1, k+1 \leq i+j, j+1 \leq i+k$ and $i+1 \leq j+k$,
then the solution $d^{(4)}$ will work and
this gives all the matrices with this $\delta$-vector.
\item If $\{i,j,k\}$ in the given vector does not satisfy any of the above cases,
there is no matrix \eqref{22}, where $\bar{*}=0$,
with this vector as its $\delta$-vector.
\end{enumerate}

Next, we consider the case with $\bar{*}=1$. By Theorem \ref{algo},
we have $\delta_{\Pc(A_4')}(t)=1+t^{k_1}+t^{k_2}+t^{k_3}$, where
\begin{align*}
k_1=1-\bigg\lfloor\frac{1-d_1-2d_1''}{4}\bigg\rfloor, \;
k_2=1-\bigg\lfloor\frac{1-d_1}{2}\bigg\rfloor \; \text{and} \;
k_3=2-\bigg\lfloor\frac{3-d_1-2d_1'}{4}\bigg\rfloor.
\end{align*}

Let $\delta_{\Pc(A_4')}(t)=1+t^i+t^j+t^k$.
We get three sets of equations, according to the order of $k_1,k_2$ and $k_3$:
\begin{align*}
i=1-\bigg\lfloor\frac{1-d_1-2d_1''}{4}\bigg\rfloor, \;
j=1-\bigg\lfloor\frac{1-d_1}{2}\bigg\rfloor \;\text{and}\;
k=2-\bigg\lfloor\frac{3-d_1-2d_1'}{4}\bigg\rfloor.
\end{align*}
or replace the roles of $i,j$ and $k$ if $i,j$ and $k$ are distinct.
After computations, considering $d_1+d_1'+d_1''$ is even, the
solutions of $(d_1,d_1',d_1'')$ are
\begin{align*}
&d^{(1)}=
\begin{cases}
d_1=2j-1 \\
d_1'=2k-j-3 \\
d_1''=2i-j-2,
\end{cases}
d^{(2)}=
\begin{cases}
d_1=2j-1 \\
d_1'=2k-j-2 \\
d_1''=2i-j-1,
\end{cases}& \\
&d^{(3)}=
\begin{cases}
d_1=2j-2 \\
d_1'=2k-j-3 \\
d_1''=2i-j-1
\end{cases}
\text{and} \;\;\; d^{(4)}=
\begin{cases}
d_1=2j-2 \\
d_1'=2k-j-2 \\
d_1''=2i-j-2.
\end{cases}&
\end{align*}
In addition, by the restriction on $(d_1,d_1',d_1'')$ that $0 \leq
d_1 \leq d-2$, $0 \leq d_1' \leq d-2$, $0 \leq d_1'' \leq d-2$,
$d_1+d_1'+d_1'' \leq 2(d-2)$, $d_1'' \leq d_1+d_1'$, $d_1' \leq
d_1+d_1''$ and $d_1 \leq d_1'+d_1''$, we have the following
characterizations:

\begin{table}[ht]\label{10-1}
\centering 
\caption{Characterizations for matrices of the form \eqref{22} with
$\bar{*}=1$}
\begin{tabular}{|c| c|c| c|c|c|c|} 
\hline 
$2k$ & $2i$ & $2j$ & $i+j$& $i+k$ & $j+k$  & solutions \\
\hline \hline  
$\ge j+3$,  & $\ge j+2$,  & $\le d-1$ & $\ge k$
& $\ge 2j+2$, & $\ge i+1$  & $d^{(1)}$\\
$\le d+j+1$ & $\le d+j$   &           &          & $\le d+1$
  &     &\\

 \hline
 $\ge j+2$, & $\ge j+1$,  & $\le d-1$ &
$\ge k$
& $\ge 2j+1$, & $\ge i+1$   & $d^{(2)}$\\
$\le d+j$ & $\le d+j-1$ & & & $\le d$ & &\\
\hline $\ge j+3$,  & $\ge j+1$,  & $\le d$ & $\ge k$
& $\ge 2j+1$, & $\ge i+2$  & $d^{(3)}$\\
$\le d+j+1$ & $\le d+j-1$ &&& $\le d+1$ & &\\
\hline $\ge j+2$,  & $\ge j+2$,  & $\le d$ & $\ge k+1$
& $\ge 2j+1$,& $\ge i+1$  & $d^{(4)}$\\
$\le d+j$& $\le d+j$ && & $\le d+1$  & &\\
\hline
\end{tabular}
\end{table}

\begin{enumerate}
\item If $j+3 \leq 2k \leq d+j+1, j+2 \leq 2i \leq d+j, 2j \leq d-1$,
$2j+2 \leq i+k \leq d+1$, $i+1 \leq j+k$ and $k \leq i+j$,
then the solution $d^{(1)}$ will work and
this gives all the matrices with this $\delta$-vector.
\item If $j+2 \leq 2k \leq d+j, j+1 \leq 2i \leq d+j-1, 2j \leq d-1$,
$2j+1 \leq i+k \leq d$, $i+1 \leq j+k$ and $k \leq i+j$,
then the solution $d^{(2)}$ will work and
this gives all the matrices with this $\delta$-vector.
\item If $j+3 \leq 2k \leq d+j+1, j+1 \leq 2i \leq d+j-1, 2j \leq d$,
$2j+1 \leq i+k \leq d+1$, $i+2 \leq j+k$ and $k \leq i+j$
then the solution $d^{(3)}$ will work and
this gives all the matrices with this $\delta$-vector.
\item If $j+2 \leq 2k \leq d+j, j+2 \leq 2i \leq d+j, 2j \leq d$,
$2j+1 \leq i+k \leq d+1$, $i+1 \leq j+k$ and $k+1 \leq i+j$
then the solution $d^{(4)}$ will work and
this gives all the matrices with this $\delta$-vector.
\item If $\{i,j,k\}$ in the given vector does not satisfy any of the above cases,
there is no matrix \eqref{22} with this vector as its $\delta$-vector.
\end{enumerate}

Notice that only the solution
\begin{align*}
d^{(3)}=
\begin{cases}
d_1=d-2 \\
d_1'=d-2 \\
d_1''=0
\end{cases}
\end{align*}
works when $i=(d+2)/4,j=d/2$ and $k=(3d+2)/4$. This happens when $d
\equiv 2 \;(\mod 4)$ and there is only one matrix with
$d_1=d_1'=d-2$. Similarly, only the solution
\begin{align*}
d^{(4)}=
\begin{cases}
d_1=d-2 \\
d_1'=0 \\
d_1''=d-2
\end{cases}
\end{align*}
works when $i=3d/4,j=d/2$ and $k=d/4+1$. This happens when $d \equiv
0 \;(\mod 4)$ and again, there is only one matrix with
$d_1=d_1''=d-2$.

\section{The classification of the possible $\delta$-vectors with $\sum_{i=0}^d\delta_i = 4$}
In this section, in consequence, we classify the possible $\delta$-vectors with
$\sum_{i=0}^d\delta_i = 4$ by results from Section \ref{vol4}.

Let $1+t^{i_1}+t^{i_2}+t^{i_3}$  with $1 \leq i_1 \leq i_2 \leq i_3 \leq d$
be a $\delta$-polynomial for some integral polytope and
$(\delta_0,\delta_1,\ldots,\delta_d)$ a sequence of the coefficients of
this polynomial, where it is clear that $\delta_0=1$ and $\sum_{i=0}^d\delta_i=4$.
Assume that $(\delta_0,\delta_1,\ldots,\delta_d)$ satisfies
the inequalities \eqref{Stanley}, \eqref{Hibi} and $\delta_1 \geq \delta_d$,
which are necessary conditions to be a possible $\delta$-vector.
Then \eqref{Stanley} and \eqref{Hibi} lead into
the following inequalities  that $(i_1,i_2,i_3)$ satisfies
\begin{align}\label{necessary}
i_3 \leq i_1+i_2,\; i_1+i_3 \leq d+1 \;\;\text{and}\;\;
i_2 \leq \lfloor(d+1)/2\rfloor.
\end{align}

Finally, the classification of possible $\delta$-vectors of integral
polytopes with $\sum_{i=0}^d\delta_i=4$ is given by the following
\begin{Theorem}\label{case4}
Let $1+t^{i_1}+t^{i_2}+t^{i_3}$ be a polynomial with $1 \leq i_1
\leq i_2 \leq i_3 \leq d$. Then there exists an integral polytope
$\Pc \subset \RR^d$ of dimension $d$ whose $\delta$-polynomial
equals
 $1+t^{i_1}+t^{i_2}+t^{i_3}$ if and only if $(i_1,i_2,i_3)$
satisfies \eqref{necessary} and an additional condition
\begin{equation}\label{add}
2i_2 \leq i_1 + i_3 \text{  or  }i_2+i_3 \leq d+1.
\end{equation}
Moreover, all these polytopes can be chosen to be simplices.
\end{Theorem}
\begin{proof}
There are four cases: (1) $i_1=i_2=i_3$, (2) $i_1<i_2=i_3$, (3)
$i_1=i_2<i_3$, (4) $i_1<i_2<i_3$. We will show that in each case
\eqref{necessary} together with \eqref{add} are the necessary and
sufficient conditions for $1+t^{i_1}+t^{i_2}+t^{i_3}$ to be the
$\delta$-vector of some integral polytope.

(1) Assume $i_1=i_2=i_3=\ell$. By the inequalities (\ref{necessary}),
we have $1 \leq \ell \leq \lfloor(d+1)/2\rfloor$.
Set $i=j=k=\ell$. We have
\begin{align}\label{condition}
j+k \geq i+1, \; 2j \leq i+k \leq d+1 \; \text{and} \; i+j \geq k+1.
\end{align} Thus, by our result on the classification of the case
of a matrix \eqref{4} (Table 2, the solution $d^{(1)}$), there
exists an integral simplex whose
$\delta$-vector is $(1,0,\ldots,0,3,0,\ldots,0)$.

On the other hand, if there exists an integral polytope with this
$\delta$-vector, then \eqref{necessary} holds since it is a
necessary condition. In this case, it follows that both inequalities
in \eqref{add} hold.

(2) Assume $\ell=i_1<i_2=i_3=\ell'$. By \eqref{necessary},
we have $1 \leq \ell < \ell' \leq \lfloor(d+1)/2\rfloor$.
Let $j=\ell$ and $i=k=\ell'$. Then the inequalities \eqref{condition} hold.
Thus there exists an integral simplex whose $\delta$-vector is
$(1,0,\ldots,0,1,0,\ldots,0,2,0,\ldots,0)$.

On the other hand, we have \eqref{necessary}.
Then, $i_2+i_3 \leq d+1$ follows from  $i_2 \leq \lfloor(d+1)/2\rfloor$.

(3) Assume $\ell=i_1=i_2<i_3=\ell'$. Set $i=\ell'$ and $j=k=\ell$.
Then it follows from \eqref{necessary} that
$$j+k \geq i, 2j+1 \leq i+k \leq d+1 \; \text{and} \; i+j \geq k+1.$$
Thus, by our result (Table 2, the solution $d^{(4)}$), there exists
an integral simplex whose $\delta$-vector is
$(1,0,\ldots,0,2,0,\ldots,0,1,0,\ldots,0)$.

On the other hand, if there exists an integral polytope with this
$\delta$-vector, then \eqref{necessary} holds.
In this case, it follows that both inequalities
in \eqref{add} hold.

(4) Assume $1\le i_1<i_2<i_3\le d$. Suppose $2i_2 \leq i_1+i_3$
holds. Set $i=i_3, j=i_2$ and $k=i_1$. Then we have $j+k=i_1+i_2
\geq i_3=i$, $2j=2i_2 \leq i_1+i_3 =i+k \leq d+1$ and $i+j=i_2+i_3
\geq 2i_2+1 \geq 2i_1+3>i_1+2=k+2.$ Thus, by our result (Table 2,
the solution $d^{(2)}$), there exists an integral simplex whose
$\delta$-vector is
$(1,0\dots,0,1,0,\dots,0,1,0,\dots,0,1,0,\dots,0)$.

Suppose  $i_2+i_3 \leq d+1$ holds. Set $i=i_3,j=i_1$ and $k=i_2$.
Then we have $j+k=i_1+i_2 \geq i_3 =i$, $2j=2i_1<i_2+i_3=i+k \leq
d+1$ and $i+j=i_1+i_3 \geq i_1+i_2+1 \geq i_2+2 = k+2$. Thus, by our
result (Table 2, the solution $d^{(2)}$), there exists an integral
simplex whose $\delta$-vector coincides with
$(\delta_0,\delta_1,\ldots,\delta_d)$.

On the other hand, assume the contrary of \eqref{add}: both $2i_2 >
i_1+i_3$ and $i_2+i_3 > d+1$ hold.
We claim that there exists no integral polytope $\Pc$ with
this $\delta$-vector. First we want to show that if there exists
such a polytope, it must be a simplex. Note that the $\delta$-vector
satisfies \eqref{necessary}. Suppose $i_1=1$. It then follows from
\eqref{necessary} and $i_2+i_3 > d+1$ that $i_2=(d+1)/2$ and
$i_3=(d+3)/2$. However, this contradicts \eqref{Hibi}. Therefore
$i_1>1$, and thus $\delta_1=0$. By an explanation after equation
\eqref{delta}, $\Pc$ must be a simplex.
Now we can apply our characteristic results for simplices.

If we set $j=i_3$, then $2j=2i_3 > i_1+i_2=i+k$. If we set $j=i_2$,
then $2j=2i_2 > i_1+i_3=i+k$. If we set $j=i_1$, then $i+k=i_2+i_3
> d+1.$ In any case there does not exist an Hermite normal form \eqref{4} whose
$\delta$-vector coincides with
$(\delta_0,\delta_1,\ldots,\delta_d)$.

Moreover, since $i+j+k=i_1+i_2+i_3 >i_2+i_3 > d+1$, there does not
exist an Hermite normal form \eqref{22} with $\bar{*}=0$ whose
$\delta$-vector coincides with
$(\delta_0,\delta_1,\ldots,\delta_d)$.

In addition, if we set $j=i_3$, then $2j=2i_3>i_1+i_2=i+k$. If we
set $j=i_2$, then $2j=2i_2>i_1+i_3=i+k$. If we set $j=i_1$, then
$i+k=i_2+i_3>d+1.$ Thus there does not exist an Hermite normal form
\eqref{22} with $\bar{*}=1$ whose $\delta$-vector coincides with
$(\delta_0,\delta_1,\ldots,\delta_d)$.
\quad\quad\quad\quad\quad\quad\quad\quad\quad
\quad\quad\quad\quad\quad\quad\quad\quad\quad\quad\quad\quad
\end{proof}

\begin{Examples}{\em
(a) We consider the integer sequence $(1,0,1,1,0,1,0)$.
Then one has $i_1=2,i_2=3,i_3=5$ and $d=6$.
Since \eqref{Stanley} and \eqref{Hibi} are satisfied and
$2i_2 \leq i_1+i_3$ holds, there is an integral polytope
whose $\delta$-vector coincides with $(1,0,1,1,0,1,0)$ by Theorem \ref{case4}.
In fact, let $M \in \ZZ^{6 \times 6}$ be the Hermite normal form \eqref{4}
with $(d_1,d_2,d_3)=(0,1,4)$ or $(0,0,5)$.
Then we have $\delta(\Pc(M))=(1,0,1,1,0,1,0)$. \\
(b) There is no integral polytope with its $\delta$-vector $(1,0,1,0,1,1,0,0)$
since we have $2i_2 > i_1+i_3$ and $i_2+i_3 > d+1$,
although this integer sequence satisfies \eqref{Stanley} and \eqref{Hibi}.
(This example is described in \cite[Example 1.2]{smallvolume}
as a counterexample of \cite[Theorem 0.1]{smallvolume}
for the case where $\sum_{i=0}^d \delta_i=4$.)
However, there exists an integral polytope
with its $\delta$-vector $(1,0,1,0,1,1,0,0,0)$
since $i_2+i_3=d+1$ holds.
}\end{Examples}

\begin{Remark}\label{simplex}{\em
From the above proof, we can see that when $\sum_{i=0}^d\delta_i=4$,
all the possible $\delta$-vectors can be obtained by simplices. This
is also true for all $\delta$-vectors with $\sum_{i=0}^d\delta_i\le 3$,
from the proof of \cite[Theorem 0.1]{smallvolume}. However,
when $\sum_{i=0}^d\delta_i=5$, the $\delta$-vector $(1,3,1)$ cannot
be obtained from any simplex, while it is a possible $\delta$-vector
of a 2-dimensional integral polygon. In fact, suppose that $(1,3,1)$
can be obtained from a simplex. Since $\min\{i:\delta_i \not= 0, i>0\}=1$
and $\max\{i:\delta_i \not= 0\}=2$, one has
$\min\{i:\delta_i \not= 0,i>0\}=3-\max\{i:\delta_i \not= 0\}$,
which implies that the assumption of \cite[Theorem 2.3]{shifted} is satisfied.
Thus the $\delta$-vector must be shifted symmetric, a contradiction.
}\end{Remark}

\end{document}